%% file: it.tex
\documentclass{article}

\usepackage{graphicx}
\usepackage{color}
\usepackage{amssymb,amsmath}
\usepackage{mathrsfs}
\usepackage{bbm}
\usepackage[text={144mm,220mm},centering]{geometry}
\usepackage[all]{xy}
\xyoption{arc}
\xyoption{rotate}

\usepackage[numbers]{natbib}
\usepackage{natbibspacing}
\renewcommand{\bibfont}{\small}

\definecolor{myurlcolor}{rgb}{0,0,0.5}
\usepackage{hyperref}
\hypersetup{colorlinks,
linkcolor=black,
citecolor=black,
urlcolor=myurlcolor}

\input{ab}

\input{th}

\author{Tom Leinster%
\thanks{School of Mathematics, University of Edinburgh, Edinburgh EH9 3JZ,
UK; Tom.Leinster@ed.ac.uk.  
Supported by an EPSRC Advanced Research Fellowship.}}
\title{\vspace*{-2em}Notions of M\"obius inversion}
\date{}

\begin{document}

\sloppy
\maketitle

\begin{abstract}
M\"obius inversion, originally a tool in number theory, was generalized to
posets for use in group theory and combinatorics.  It was later generalized
to categories in two different ways, both of which are useful.  We provide a
unifying abstract framework.  This allows us to compare and contrast the two
theories of M\"obius inversion for categories, and advance each of them.
Among several side benefits is an improved understanding of the following
fact: the Euler characteristic of the classifying space of a (suitably finite)
category depends only on its underlying graph.
\end{abstract}

\tableofcontents\ 

\section*{Introduction}
\ucontents{section}{Introduction}

The history of M\"obius inversion begins with August Ferdinand M\"obius
(1790--1868), the basic aspects of whose work on this can be described in
modern terms as follows.  Consider sequences $\alpha(1), \alpha(2), \ldots$ of
complex numbers.  Any two sequences $\alpha$, $\beta$ have a convolution
product $\alpha \conv \beta$, defined by
\[
(\alpha\conv\beta)(n) = \sum_{k, m\cln km=n} \alpha(k) \beta(m).
\]
This product has a unit, and the constant sequence $\zeta = (1, 1, \ldots)$
has a convolution inverse: the classical M\"obius function $\mu$, given by a
well-known formula involving prime factorizations.  It has many uses in
elementary and not-so-elementary number theory.  For example, every 
sequence $\alpha$ determines a formal Dirichlet series $\sum_{n =
1}^\infty \alpha(n)/n^s$, where $s$ is a formal variable.  Convolution of
sequences corresponds to multiplication of Dirichlet series.  The constant
sequence $\zeta$ corresponds to the Riemann zeta function, and the
relationship between $\zeta$ and $\mu$ can be expressed as
\[
\sum_{n = 1}^\infty \frac{1}{n^s} 
=
1
\bigg/
\sum_{n = 1}^\infty \frac{\mu(n)}{n^s}.
\]

In the mid-twentieth century, it was realized that M\"obius inversion could
usefully be defined for general partially ordered sets, the original case
being the set of positive integers ordered by divisibility.  This insight
is usually associated with the name of Gian-Carlo Rota~\cite{Rota}.  Although
Rota was not (as he made clear) the first to generalize M\"obius inversion to
posets, he was responsible for harnessing its power to solve problems in
enumerative combinatorics.  

Rota's theory was subsequently generalized by multiple people in multiple
directions, but two particularly concern us here.  Both are theories of
M\"obius inversion for categories (Fig.~\ref{fig:history}).

\begin{figure}
\centering
\setlength{\unitlength}{1mm}
\begin{picture}(122,50)(-60,8)
\thicklines
\cell{0}{12}{c}{\includegraphics[width=25\unitlength]{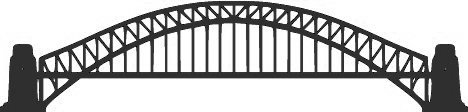}}
\cell{0}{55}{c}{Number-theoretic M\"obius inversion}
\cell{0}{50}{c}{(M\"obius 1832)}
\put(0,46){\vector(0,-1){7}}
\cell{0}{35}{c}{M\"obius inversion for posets}
\cell{0}{30}{c}{(Rota 1964, et al.)}
\put(-15,26){\vector(-1,-1){7}}
\cell{-43}{15}{c}{Fine M\"obius inversion for categories}
\cell{-43}{10}{c}{(Leroux et al.\ 1975, '80; Haigh 1980)}
\put(15,26){\vector(1,-1){7}}
\cell{43}{15}{c}{Coarse M\"obius inversion for categories}
\cell{43}{10}{c}{(Haigh 1980; Leinster 2008)}
\end{picture}
\caption{Simplified history of M\"obius inversion.  This paper builds a
bridge between the two notions of M\"obius inversion for categories.  Notably
missing from the diagram are the finite difference calculus and the theories
of M\"obius inversion developed by Cartier and Foata~\cite{CaFo},
D\"ur~\cite{Duer} and L\"uck~\cite{Luec}.}
\label{fig:history}
\end{figure}
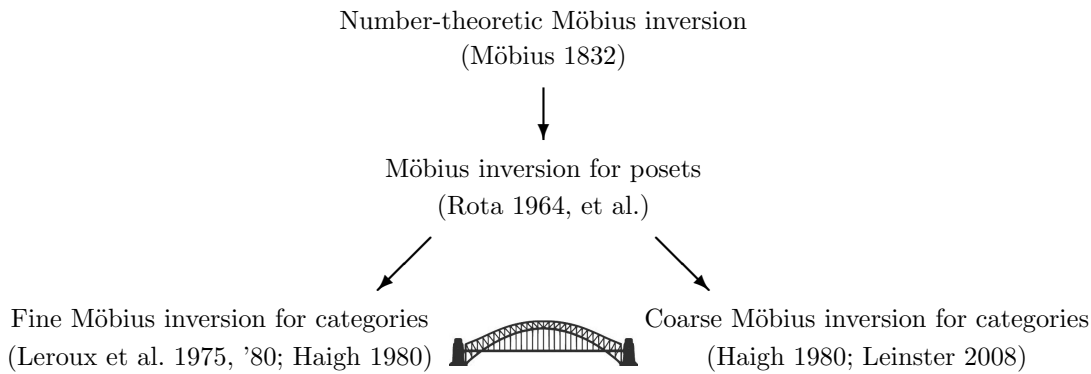

The first was developed independently by Pierre Leroux and collaborators and
by John Haigh.  (Leroux published a short announcement in 1975~\cite{Lero}.
The full account, joint with Content and Lemay, appeared in 1980~\cite{CLL},
as did Haigh's paper~\cite{Haig}.)  The second was also introduced by Haigh
(in Section~3 of~\cite{Haig}), in just a dozen lines of text.  It was
developed more fully by the author~\cite{ECC} as part of the theory of
Euler characteristic of categories.

A comparison of the two theories would be hopelessly confusing if both were
referred to as `M\"obius inversion'.  We therefore introduce new terminology.
The first type of M\"obius inversion is called `fine', and the second
`coarse'.  These same adjectives are applied systematically throughout; for
example, the finiteness condition used in the fine theory is renamed `fine
finiteness', and its coarse counterpart `coarse finiteness'.  This makes
various relationships clear.  The names are apt: the fine M\"obius function of
a category is a more refined invariant, more sensitive to the category's
structure than the coarse one.  And there are far more categories for which
the coarse M\"obius function is well-defined than the fine one: it is like a
weed that grows almost anywhere, compared to a fine but delicate flower.  See
Examples~\ref{eg:group-fine} and~\ref{eg:monoid-coarse}, and
Theorem~\ref{thm:comparison-fin}.

This paper proves results connecting the two theories, points out essential
differences between them, and advances each one further.  But more
importantly, it provides a single abstract setting in which all of this takes
place.  As we shall see, the two theories, together with a third intermediate
one, arise from the inclusions of categories
\[
\One \incl \Two \incl \Set
\]
in a uniform manner (Fig.~2).

We begin with a review of fine and coarse M\"obius inversion for categories,
introducing the new terminology (Section~\ref{sec:fac}).  The basic theorem
connecting them is stated.  We see that the fine and coarse theories can help
each other: for instance, Corollaries~\ref{cor:Menni},
\ref{cor:Menni-Euler} and~\ref{cor:Menni-pf} are all stated purely in terms of
fine M\"obius inversion, but proved using the coarse theory.

We also explore in Section~\ref{sec:fac} the following curious phenomenon.
Every small category gives rise to a topological space, its classifying space
or geometric realization.  It is a fact that (under finiteness assumptions)
the Euler characteristic of that space is independent of the composition in
the category.  The theory of coarse M\"obius inversion sheds light on this.

The abstract framework is introduced in Section~\ref{sec:fun}.  The key is the
covariant functoriality of the incidence algebra construction.  As the
framework is developed, a third level naturally emerges, between coarse and
fine.  This allows the coarse theory, previously confined for the most part to
finite categories, to be extended to infinite categories
(Section~\ref{sec:inf}).  There we generalize one of the main theorems of
Rota's original paper~\cite{Rota}.

Coarse M\"obius inversion also makes sense for enriched categories
(Section~\ref{sec:enr}).  This fact has already been
exploited in investigations of geometric measure in metric spaces, as will be
explained. 

The remaining sections are contributions to the fine theory.  The incidence
algebra construction is functorial in both the covariant and contravariant
senses, and in Section~\ref{sec:fun-re}, we prove a Beck--Chevalley theorem
enabling the two to be unified.  In Section~\ref{sec:McL}, we prove a
new characterization of the `M\"obius categories' of Leroux.

There are two appendices.  Setting up the coarse theory for infinite
categories requires a nontrivial result on inverse matrices, proved in
Appendix~\ref{app:zeros}.  Finally, Appendix~\ref{app:pbh} creates an abstract
home for the notion of functor with unique lifting of factorizations,
important for M\"obius inversion.  The concept developed there,
`pullback-homomorphism', may also be of more general interest.

\paragraph*{Related work}  I will not attempt to survey the large body of work
on M\"obius inversion for posets; see~\cite{StanEC1} for a good overview.
Lawvere and Menni's paper~\cite{LaMe} contains further pointers to the
literature on fine M\"obius inversion for categories.  Coarse M\"obius
inversion is used in the theory of the Euler characteristic of a category,
which since the original paper~\cite{ECC} has been developed and applied by
Berger and Leinster~\cite{ECCSDS}, Fiore, L\"uck and Sauer~\cite{FLSECC,
FLSFOE}, Jacobsen and M{\o}ller~\cite{JaMoe}, and Noguchi~\cite{NoguECA,
NoguECC, NoguZFFS, NoguZFFM}.  Sections~\ref{sec:inf} and~\ref{sec:enr}
of the present work expand on points covered briefly in Sections~4 and~2,
respectively, of~\cite{ECC}.  We do not touch here on the theory of
M\"obius inversion developed by Cartier and Foata~\cite{CaFo} for use in
combinatorics, nor that of D\"ur~\cite{Duer} or L\"uck~\cite{Luec}.

Many of the finiteness conditions arising in M\"obius inversion for categories
were explored by Mitchell~\cite{Mitc}, as was the incidence algebra
construction.  The question of which finite directed graphs admit a category
structure, implicitly raised by Lemma~\ref{lemma:trans}, has recently been
answered by Allouch~\cite{AlloECAP, AlloECA}.

\paragraph*{Notation} Given a small category $\A$, we write $\obc{\A}$ for its
set of objects and $\arr{\A}$ for the set of all maps or
morphisms in $\A$.  We often write $a \in \A$ to mean $a \in
\obc{\A}$, and we write $\A(a, b)$ for the set of maps from $a$
to $b$.  Given a finite set $X$, we write $\card{X}$ for its
cardinality.

\paragraph*{Acknowledgements} I thank John Baez, Nathan Bowler, Joachim
Kock, Mat\'ias Menni, Mike Shulman, Todd Trimble and Russ Woodroofe for
useful and enlightening conversations.  I am also grateful for the comments
of the anonymous referee.

\section{Fine and coarse M\"obius inversion}
\label{sec:fac}

Here we review the two types of M\"obius inversion for categories,
introducing systematic new terminology.  We then state the basic result
connecting the fine and coarse theories.

A \demph{rig} (or semiring) is a ri\emph{n}g without \emph{n}egatives: a set
equipped with a commutative monoid structure $(+, 0)$ and a monoid structure
$(\cdot, 1)$, the latter distributing over the former.  We take rig to mean
\emph{commutative} rig: one whose multiplication is commutative.  Similarly,
\demph{ring} means commutative ring.  A rig is \demph{trivial} if it has only
one element.  For a natural number $n$ and a rig $k$, we often use $n$ to
denote the element $n\cdot 1 = 1 + \cdots + 1$ of $k$.

A \demph{module} over a rig $k$ is a commutative monoid acted on by $k$, in
the evident sense; \demph{algebras} over $k$ are defined similarly (and not
assumed to be commutative).  When $k$ is a ring, $k$-modules are the same
whether $k$ is regarded as a rig or as a ring.  The same goes for algebras.

\subsection*{Fine M\"obius inversion}

The convolution of the opening paragraph involved a sum, and that sum
is well-defined because each positive integer has only finitely many
factorizations into two parts.  Similarly, when developing M\"obius
inversion for categories, we need to impose the following finiteness
condition.  A category $\A$ is \demph{finely finite} if for each map $f\cln a
\to b$ in $\A$, there are only finitely many diagrams
\[
a \toby{g} c \toby{h} b
\]
in $\A$ whose composite is $f$.

Let $\A$ be a finely finite category and $k$ a rig.  The \demph{fine incidence
algebra} $\iaf{k}{\A}$ is the set of all functions $\arr{\A} \to k$, made into
a $k$-algebra as follows.  Its $k$-module structure is pointwise.  The
multiplication $\conv$ is given by
\[
(\alpha\conv\beta)(f)
=
\sum_{h g = f}
\alpha(g) \, \beta(h),
\]
where $\alpha, \beta \in \iaf{k}{\A}$ and $f \in \arr{\A}$.  (Fine finiteness
guarantees that the sum is finite.)  The multiplicative unit $\delta$ is given
by $\delta(1_a) = 1$ whenever $a \in \A$, and $\delta(f) = 0$ otherwise.

The fine incidence algebra has a special element: the
\demph{fine zeta function} $\zeta_\A$, defined by
\[
\zeta_\A(f) = 1
\]
for all $f \in \arr{\A}$.  We say that $\A$ has \demph{fine M\"obius
inversion} over $k$ if $\zeta_\A$ has a multiplicative inverse in
$\iaf{k}{\A}$, which is called the \demph{fine M\"obius function}
$\mu_\A = \zeta_\A^{-1} \in \iaf{k}{\A}$.  

These terms are all new; let us compare them with previous usage.  Where we
call a category finely finite, Leroux et al.~\cite{CLL} say that it `has
finite decompositions of degree 2'.  What we call the fine incidence algebra
and fine M\"obius function, they simply call the incidence algebra and
M\"obius function.  They also have a definition of `M\"obius category'.  Being
a M\"obius category is a stronger condition than having fine M\"obius
inversion.  The precise relationship is determined in Section~\ref{sec:McL},
but we will not need the concept of M\"obius category elsewhere.

Haigh~\cite{Haig} removes the possibility of infinite sums by a different
strategy: he imposes no finiteness conditions on $\A$, but considers only
those functions $\alpha\cln \arr{\A} \to k$ such that $\alpha(f) = 0$ for all
but finitely many maps $f$.  He calls the resulting algebra the `category
algebra'; it only has a multiplicative identity if $\A$ is finite.  He
calls a finite category $\A$ a `M\"obius category' if it has fine M\"obius
inversion, in conflict with the usage of Leroux et al.

Both Haigh and Leroux et al.\ take $k$ to be a ring, not a general rig.

\begin{example} \label{eg:poset-fine}
Let $A$ be a partially ordered set, viewed as a category.  It is finely finite
if and only if it is \demph{locally finite}: for all $a, b \in A$, the set $\{
c \in A \such a \leq c \leq b \}$ is finite.  (Rota's theory proceeds on this
assumption.) The fine incidence algebra is the set of functions
\[
\{ (a, b) \in A \times A \such a \leq b \}
\to 
k,
\]
and the fine M\"obius function $\mu_A$, if it exists, is characterized by the
equations
\[
\sum_{c\cln a \leq c \leq b} \mu_A(a, c) 
=
\sum_{c\cln a \leq c \leq b} \mu_A(c, b)
=
\begin{cases}
1       &\text{if } a = b       \\
0       &\text{otherwise}
\end{cases}
\]
($a, b \in A$).  Hall~\cite{Hall} showed that when $k$ is a ring, the
M\"obius function exists and is given by
\[
\mu_A(a, b)
=
\sum_{n \in \nat}
(-1)^n
\cdot
\card{\{ \text{chains } a = a_0 < \cdots < a_n = b \}}.
\]
For example, if $A$ is the poset of positive integers ordered by divisibility,
and $k = \integers$, then $\mu_A(a, b) = \mu(b/a)$, where the $\mu$ on the
right-hand side is the classical M\"obius function.
\end{example}

\begin{example} \label{eg:group-fine}
A group, viewed as a one-object category, is finely finite if and only if
it is finite.  The fine incidence algebra $\iaf{k}{G}$ of a finite group $G$
is its group algebra.  No group has fine M\"obius inversion, except when it or
$k$ is trivial.
\end{example}

\subsection*{Coarse M\"obius inversion}

A category is \demph{finite}---or \demph{coarsely finite}, for emphasis---if
it has only finitely many objects and arrows.  For now, coarse M\"obius
inversion will be defined only for finite categories.  We will see how to
relax this assumption in Section~\ref{sec:inf}.

Let $\A$ be a finite category and $k$ a rig.  The \demph{coarse incidence
algebra} $\iac{k}{\A}$ is the set of all functions $\obc{\A} \times \obc{\A}
\to k$, made into a $k$-algebra as follows.  Its $k$-module structure is
pointwise.  The multiplication $\conv$ is given by
\[
(\alpha\conv\beta)(a, b)
=
\sum_{c \in \A}
\alpha(a, c) \, \beta(c, b)
\]
($\alpha, \beta \in \iac{k}{\A}, a, b \in \A$).  The multiplicative unit is
the Kronecker $\delta$, defined by $\delta(a, b) = 1$ if $a = b$ and
$\delta(a, b) = 0$ otherwise.  If a total order is chosen on the $n$ objects
of $\A$, then $\iac{k}{\A}$ is just the algebra of $n\times n$ matrices over
$k$.

The coarse incidence algebra has a special element: the \demph{coarse zeta
function} $\zeta_\A$, defined for $a, b \in \A$ by 
\[
\zeta_\A(a, b) 
=
\card{(\A(a, b))}
\in k.
\]
We say that $\A$ has \demph{coarse M\"obius inversion} over $k$ if $\zeta_\A$
has a multiplicative inverse in $\iac{k}{\A}$.  The \demph{coarse M\"obius
function} is then $\mu_\A = \zeta_\A^{-1} \in \iac{k}{\A}$.

In \cite{ECC}, the algebra $\iac{k}{\A}$ is only considered in the case $k =
\rationals$.  What we call coarse M\"obius inversion and the coarse M\"obius
function here are simply called M\"obius inversion and the M\"obius function
there.  The same is true in \cite{ECCSDS} and \cite{MMS}.

\begin{example} \label{eg:poset-coarse}
Let $A$ be a finite partially ordered set.  The coarse incidence algebra is
the set of functions $A \times A \to k$.  It contains the fine incidence
algebra as a subalgebra, consisting of those $\alpha\cln A \times A \to k$ such
that $\alpha(a, b) = 0$ whenever $a \not\leq b$.  The fine and coarse zeta
functions, viewed as elements of the coarse incidence algebra, are equal.
Hence when $A$ has fine M\"obius inversion (e.g.\ when $k$ is a ring), the
fine and coarse M\"obius functions are also equal.
\end{example}

No confusion should be caused by writing $\zeta_\A$ for both the fine and
coarse zeta functions.  When we write `$\zeta_\A(f)$', the $\zeta_\A$ in
question must be the fine one; when we write `$\zeta_\A(a, b)$', it must be
the coarse one.  \latin{A priori} there could be an ambiguity when $\A$ is a
poset, since there we might use $(a, b)$ to denote the unique map $a \to b$.
But the previous example shows that in that case, the two meanings of
$\zeta_\A(a, b)$ agree.  The same goes for the fine and coarse M\"obius
functions $\mu_\A$.  Moreover, when $\A$ is understood, we write them as just
$\zeta$ and $\mu$.

\begin{example} \label{eg:monoid-coarse}
Let $M$ be a finite monoid.  Then $\iac{k}{M} = k$, and $\zeta =
\card{M}\in k$.  So, for instance, if $k$ is a field of characteristic $0$
then every finite monoid has coarse M\"obius inversion over $k$.  Contrast
Example~\ref{eg:group-fine}.
\end{example}

A category with coarse M\"obius inversion over a nontrivial rig must be
\demph{skeletal}, that is, isomorphic objects must be equal.  (For if not, the
matrix $\zeta$ would have two identical rows.)  But since every category is
equivalent to some skeletal category, this is not a serious restriction.
Large classes of finite skeletal categories have coarse M\"obius inversion
over $\rationals$: all posets, groupoids and monoids, all categories 
containing no nontrivial idempotents, and all categories admitting an epi-mono
factorization system.  See~\cite{ECC} for details.

The \demph{Euler characteristic} of a finite category $\A$ with coarse
M\"obius inversion is $\chi(\A) = \sum_{a, b \in \A} \mu_\A(a, b)$.  (This can
be taken as a definition, although in fact Euler characteristic can be defined
under weaker hypotheses~\cite{ECC}.)  The name is largely justified by the
following fact.  Let $\A$ be a finite skeletal category containing no
nontrivial endomorphisms.  Write $\gr{N\A}$ for its classifying space, that
is, the geometric realization of its simplicial nerve $N\A$.  Proposition~2.11
of~\cite{ECC} states that $\chi(\A) = \chi(\gr{N\A})$.  Further results
in~\cite{ECC} relate the Euler characteristic of categories to other
invariants of size: the Euler characteristics of graphs, posets and orbifolds,
the cardinality of sets, and the Baez--Dolan cardinality of
groupoids~\cite{BaDoFSF}.

The coarse M\"obius function of a category does not depend on its composition,
just its underlying directed graph.  The same is therefore true of Euler
characteristic.  Of course, every nontrivial invariant throws away \emph{some}
information, but to throw away the composition of a category might be thought
extravagant.

Nevertheless, there is an important precedent.  Consider homotopically tame
spaces---say, finite CW-complexes.  Any such space $X$ can be built up from a
stock of points, intervals, disks, etc., by gluing them together, and it
hardly needs saying that the topology of $X$ depends entirely on \emph{how}
they are glued together.  But the Euler characteristic does not.
Topologically important as Euler characteristic is, it is independent of
gluing.  

The result on classifying spaces implies:
\begin{propn}   \label{propn:cs-indt}
Let $\A$ and $\A'$ be finite skeletal categories containing no nontrivial
endomorphisms.  If they have the same underlying directed graph then
$\chi(\gr{N\A}) = \chi(\gr{N\A'})$.%
\done
\end{propn}

Now, the theory of group homology is set up so that the homology of a group
is equal to the homology of its classifying space.  If we wish the
analogous statement to be true of Euler characteristic of categories (under
finiteness conditions), Proposition~\ref{propn:cs-indt} \emph{forces} it to
be independent of composition.

One could, nonetheless, develop the theories of coarse M\"obius inversion and
Euler characteristic for arbitrary directed graphs.  Many of the results
in~\cite{ECC} and~\cite{ECCSDS} involve categorical concepts: automorphisms,
epi-mono factorization, equivalences, adjunctions, fibrations, \ldots.  In
principle, it must be possible to rephrase them purely in terms of graphs, but
it is not yet clear that it is fruitful to do so.  Perhaps the following
situation is comparable.  Limits in a category $\cat{C}$ are usually phrased
in terms of a functor $\scat{I} \to \cat{C}$, even though the definition of
limit does not use the category structure on $\scat{I}$.  One could therefore
rephrase all results about limits in terms of graphs $\scat{I}$; but it is not
clear that this is a useful step to take.

\subsection*{Comparison between fine and coarse}

In the interests of describing the relationship between fine and coarse
M\"obius inversion as soon as possible, we first state a result under
unnecessarily restrictive hypotheses.  It first appeared as Proposition~3.6 of
Haigh~\cite{Haig}, and was also stated at the end of Section~4 of \cite{ECC}.
The unrestricted form appears as Theorem~\ref{thm:comparison-pf} below.

Fix a rig $k$.  

\begin{thm}[Haigh]     \label{thm:comparison-fin}
Let $\A$ be a finite category.  If $\A$ has fine M\"obius inversion over $k$
then $\A$ also has coarse M\"obius inversion over $k$, given for $a, b \in \A$ 
by
\[
\mu_\A(a, b) 
=
\sum_{f \in \A(a, b)} \mu_\A(f).
\]
\end{thm}

This can easily be proved by a direct calculation, but a proof also arises
naturally in our abstract development (Section~\ref{sec:fun}).

The following corollary is due to Mat\'ias Menni (private communication,
2010).

\begin{cor}[Menni]      \label{cor:Menni}
Let $\A$ and $\scat{A'}$ be finite categories with fine M\"obius inversion
over $k$.  Suppose that $\A$ and $\scat{A'}$ have the same underlying directed
graph.  Then for all objects $a, b$,
\[
\sum_{f\cln a \to b} \mu_{\A}(f) 
= 
\sum_{f\cln a \to b} \mu_{\scat{A'}}(f).
\]
\end{cor}

\begin{proof}
By Theorem~\ref{thm:comparison-fin}, $\A$ and $\scat{A'}$ have coarse M\"obius
inversion and the equation is equivalent to $\mu_{\A}(a, b) =
\mu_{\scat{A'}}(a, b)$.  This is true because the coarse M\"obius function of
a category depends only on its underlying graph.  \done
\end{proof}

\begin{cor}     \label{cor:Menni-Euler}
Let $\A$ and $\scat{A'}$ be finite categories with fine M\"obius inversion
over $k$.  Suppose that $\A$ and $\scat{A'}$ have the same underlying directed
graph.  Then
\marginpar{%
\setlength{\unitlength}{1mm}%
\begin{picture}(0,0)\put(-6.5,-8.5){$\Box$}\end{picture}}
\[
\sum_{f \in \arr{\A}} \mu_{\A}(f) 
= 
\sum_{f \in \arr{\scat{A}'}} \mu_{\scat{A'}}(f).
\]
\end{cor}

The two sides of this equation are the Euler characteristics of $\A$ and
$\scat{A'}$.  But note that Corollaries~\ref{cor:Menni}
and~\ref{cor:Menni-Euler}, while proved using the coarse theory, refer only
to the theory of \emph{fine} M\"obius inversion.

\section{Functoriality}
\label{sec:fun}

We have seen that each sufficiently finite category $\A$ gives rise to a
$k$-algebra $\iaf{k}{\A}$, for each rig $k$.  Here we show how this process
can be made functorial in $\A$.  Although we deal primarily with fine
incidence algebras, the coarse ones enter naturally as the story unfolds.

There is a well-known way to make $\A \mapsto \iaf{k}{\A}$ functorial in the
\emph{contravariant} sense, using functors with unique lifting of
factorizations.  This is discussed in Section~\ref{sec:fun-re}, but is not
needed to achieve the main aims of this paper.  Instead, we make $\A \mapsto
\iaf{k}{\A}$ into a \emph{covariant} functor.

Let $\A$ and $\B$ be finely finite categories.  Let $F\cln \A \to \B$ be a
functor with \demph{finite fibres}, meaning that for each $g \in \arr{\B}$,
the set $\{ f \in \arr{\A} \such Ff = g\}$ is finite.  (This implies the
analogous condition on objects.)  There is an induced $k$-linear map
\[
F_!\cln \iaf{k}{\A} \to \iaf{k}{\B}
\]
defined for $\alpha \in \iaf{k}{\A}$ and $g \in \arr{\B}$ by 
\[
(F_! \alpha)(g) 
=
\sum_{f\cln Ff = g}
\alpha(f).
\]

This covariant functoriality was introduced by Content, Lemay and
Leroux~\cite{CLL}.  The following result is close to their Proposition~5.6.

\begin{propn}   \label{propn:fun}
Let $\A$ and $\B$ be finely finite categories, and let $F\cln \A \to \B$ be a
functor with finite fibres.  Then $F_!\cln \iaf{k}{\A} \to \iaf{k}{\B}$ is an
algebra homomorphism for all rigs $k$ if and only if $F$ is bijective on
objects. 
\end{propn}

\begin{proof}
First consider preservation of identities.  For each map $g$ in $\scat{B}$, we
have 
\[
(F_! \delta_\A)(g)
=
\sum_{f\cln Ff = g} \delta_\A(f)
=
\card\{a \in \obc{\A} \such 1_{F(a)} = g\} \in k.
\]
If $g$ is not an identity then $(F_! \delta_\A)(g) = 0 = \delta_\B(g)$.  If
$g$ is an identity, say $g = 1_b$, then 
\[
(F_! \delta_\A)(1_b) 
=
\card{\{a \in \obc{\A} \such Fa = b\}}
\]
and $\delta_\B(1_b) = 1$.  Hence if $F$ is bijective on objects then $F_!
\delta_\A = \delta_\B$.  Conversely, if $F_!  \delta_\A = \delta_\B$ for $k =
\integers$ then $F$ is bijective on objects.

A straightforward calculation shows that if $F$ is injective on objects then
$F_!$ preserves binary multiplication.  \done
\end{proof}

Write $\Catboff$ for the category whose objects are finely finite categories
and whose maps are bijective-on-objects functors with finite fibres.  There is
a functor $\Catboff \to k\hyph\Alg$ defined by $\A \mapsto \iaf{k}{\A}$ and $F
\mapsto F_!$.

\begin{example} \label{eg:fun-co}
Given a category $\A$, write $\codisc{\A}$ for the codiscrete category with
the same objects as $\A$; thus, there is precisely one map $a \to b$ in
$\codisc{\A}$ for each pair $(a, b)$ of objects.  There is a unique
identity-on-objects functor $\A \to \codisc{\A}$.  Assume now that $\A$ is
(coarsely) finite.  Then $C\A$ is finely finite and $\A \to C\A$ has finite
fibres.

The coarse incidence algebra of $\A$ is the fine incidence algebra
of the codiscrete category on $\A$: 
\[
\iac{k}{\A} \iso \iaf{k}{(\codisc{\A})}.
\]
So the functor $\A \to \codisc{\A}$ induces a homomorphism of
$k$-algebras
\[
\Sigma\cln \iaf{k}{\A} \to \iac{k}{\A}.
\]
Explicitly,
\begin{equation}        \label{eq:Sigma-def}
(\Sigma \alpha)(a, b)
=
\sum_{f \in \A(a, b)} \alpha(f)
\end{equation}
($\alpha \in \iaf{k}{\A}$, $a, b \in \A$).  The image under $\Sigma$ of
the fine zeta function $\zeta_\A \in \iaf{k}{\A}$ is the coarse zeta function
$\zeta_\A \in \iac{k}{\A}$.  This proves Haigh's
comparison theorem:
\end{example}

\begin{prooflike}{Proof of Theorem~\ref{thm:comparison-fin}}
$\Sigma\cln \iaf{k}{\A} \to \iac{k}{\A}$ is an algebra homomorphism mapping
$\zeta_\A \in \iaf{k}{\A}$ to $\zeta_\A \in \iac{k}{\A}$, so it also maps
$\mu_\A = \zeta_\A^{-1} \in \iaf{k}{\A}$ to $\mu_\A = \zeta_\A^{-1} \in
\iac{k}{\A}$.  \done
\end{prooflike}

A \demph{preorder} on a set is a reflexive transitive binary relation.  The
2-categories of preordered and partially ordered sets are equivalent, so the
difference between the two types of structure is inessential; both will be
referred to as `posets'.

\begin{example} \label{eg:fun-po}
Let $\A$ be a small category.  There is a preorder on the set of objects of
$\A$ defined by $a \leq b$ if and only if there is at least one map $a \to
b$.  Denote the resulting poset by $\poset{\A}$.  There is a unique
identity-on-objects functor $\A \to \poset{\A}$.  

In order for this to induce a homomorphism $\iaf{k}{\A} \to
\iaf{k}{(\poset{\A})}$, and in order for the algebras $\iaf{k}{\A}$ and
$\iaf{k}{(\poset{\A})}$ to be defined at all, some finiteness conditions must
hold.  We defer precise discussion of those conditions to the next section,
temporarily making the simplifying assumption that $\A$ is coarsely finite.
This suffices.

Write
\[
\iap{k}{\A} = \iaf{k}{(\poset{\A})}.
\]
Thus, $\iap{k}{\A}$ consists of the functions $\{(a, b) \in \obc{\A} \times
\obc{\A} \such \A(a, b) \neq \emptyset\} \to k$.  It can also be seen as a
subalgebra of $\iac{k}{\A}$: 
\begin{equation}        \label{eq:iap-sub}
\iap{k}{\A} 
\iso
\{ \alpha \in \iac{k}{\A}
\such
\A(a, b) = \emptyset \implies \alpha(a, b) = 0 \}.
\end{equation}
The functor $\A \to \poset{\A}$ induces a homomorphism $\Sigma\cln \iaf{k}{\A}
\to \iap{k}{\A}$, given by equation~\eqref{eq:Sigma-def} above.  So we have a
commutative triangle of $k$-algebras as in Fig.~\ref{fig:triangles}(d).
\end{example}

\begin{figure}
\centering
\vspace*{-6mm}
\begin{tabular}{c@{\qquad}c@{\qquad}c@{\qquad}c}
\begin{xy}
(10,0)*+{\,\One}="br";
(-10,0)*+{\Two\,}="bl";
(0,14)*+{\Set}="top";
{\ar@<1.2ex>@{_{(}->} "br";"bl"};
{\ar@/^.4pc/_{\rotatebox{90}{$\scriptstyle\!\!\!\vdash$}} "bl";"br"};
{\ar@<-1.2ex>@{^{(}->} "br";"top"};
{\ar@/_.4pc/^{\rotatebox{35}{$\scriptstyle\!\!\!\dashv$}} "top";"br"};
{\ar@<1ex>@{_{(}->}_{\rotatebox{-35}{$\scriptstyle\!\!\vdash$}} "bl";"top"};
{\ar@/^.4pc/ "top";"bl"};
\end{xy}
&
\begin{xy}
(10,0)*+{\,\Set}="br";
(-10,0)*+{\Poset}="bl";
(0,14)*+{\Cat}="top";
{\ar@<1ex>@{_{(}->} "br";"bl"};
{\ar@/^.4pc/_{\rotatebox{90}{$\scriptstyle\!\!\!\vdash$}} "bl";"br"};
{\ar@<-1ex>@{^{(}->} "br";"top"};
{\ar@/_.4pc/^{\rotatebox{35}{$\scriptstyle\!\!\!\dashv$}} "top";"br"};
{\ar@<1ex>@{_{(}->}_{\rotatebox{-35}{$\scriptstyle\!\!\vdash$}} "bl";"top"};
{\ar@/^.4pc/ "top";"bl"};
\end{xy}
&
\begin{xy}
(10,0)*+{C\A}="br";
(-10,0)*+{P\A}="bl";
(0,14)*+{\A}="top";
{\ar@<-.5ex>@{^{(}->} "bl";"br"};
{\ar_{\text{unit}} "top";"bl"};
{\ar^{\text{unit}} "top";"br"};
\end{xy}
&
\begin{xy}
(10,0)*+{\iac{k}{\A}}="br";
(-10,0)*+{\iap{k}{\A}}="bl";
(0,14)*+{\iaf{k}{\A}}="top";
{\ar@<-.5ex>@{^{(}->} "bl";"br"};
{\ar_{\Sigma} "top";"bl"};
{\ar^{\Sigma} "top";"br"};
\end{xy}\\[2ex]
(a) &(b) &(c) &(d)
\end{tabular}
\caption{Abstract origins of the three incidence algebras.}
\label{fig:triangles}
\end{figure}
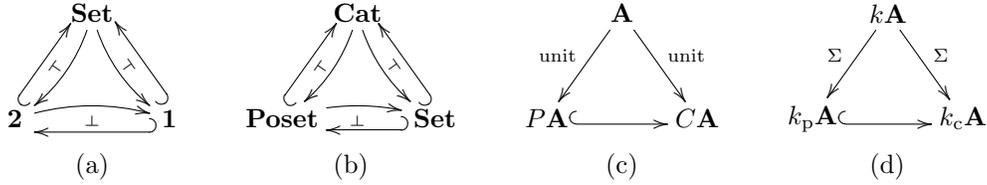

This commutative triangle also arises inexorably from very simple origins, by
applying standard categorical constructions.  We start with the inclusions of
categories
\[
\One \incl \Two \incl \Set.
\]
Here $\Two$ is the full subcategory of $\Set$ consisting of the empty set
$\emptyset$ and the one-element set $1$, and $\One$ is the subcategory
consisting of $1$ alone.  Both inclusions have left adjoints, giving the
commutative triangle of Fig.~\ref{fig:triangles}(a).  Moreover, all the
categories have finite products and all the functors preserve them.  So we may
apply the 2-functor $\cat{V} \mapsto \cat{V}\hyph\Cat$, giving the commutative
triangle of Fig.~\ref{fig:triangles}(b).  The adjunction $\Poset \oppairu
\Cat$ induces a monad $P$ on $\Cat$, the adjunction $\Set \oppairu \Cat$
induces a monad $C$ on $\Cat$, and the adjunction $\Poset \oppairu \Set$
induces a map of monads $P \to C$.  So for each $\A \in \Cat$, we obtain a
commutative triangle as in Fig.~\ref{fig:triangles}(c).

The categories $P\A$ and $C\A$, and all three functors, are the same as in the
explicit descriptions above.  In particular, the functors are bijective on
objects.  So assuming that $\A$ is finite, we may take fine
incidence algebras throughout, and the result is the commutative
triangle of Fig.~\ref{fig:triangles}(d).

\section{M\"obius inversion for infinite categories}
\label{sec:inf}

Here we extend the theory of coarse M\"obius inversion to a class of infinite
categories.  The relationship between coarse and fine M\"obius inversion,
stated in Theorem~\ref{thm:comparison-fin}, persists. 

Fix a rig $k$.  Assume that $k$ has \demph{characteristic zero}, in the sense
that $0$ is the only natural number $n$ satisfying $n\cdot 1 = 0 \in k$.

The finiteness condition that we are about to introduce can be motivated both
pragmatically and abstractly.  

Pragmatically, we seek the minimal finiteness conditions on a category $\A$
allowing the apparatus of coarse M\"obius inversion to be set up.
First, for $\zeta_\A$ to make sense, the homsets of $\A$ must be finite.
Second, if $\zeta_\A$ is to belong to an incidence algebra with the usual kind
of convolution product, then in particular $\zeta_\A \conv \zeta_\A$ must be
defined; and since we have no way of handling infinite sums, we require that
for each $a, b \in \A$, there are only finitely many $c \in \A$ such that
$\zeta_\A(a, c) \zeta_\A(c, b) \neq 0$.  For that to be true over all
rigs, for each $a, b$ there can be only finitely many $c$ such that there
exist maps $a \to c \to b$.

We will see that these two requirements suffice.

\begin{defn}
Let $a$ and $b$ be objects of a category $\A$.  The \demph{patch}
$\patch{\A}{a}{b}$ is the full subcategory of $\A$ with objects
$\{c \in \A \such \text{there exist maps } a \to c \to b \}$.
\end{defn}
(A patch might also be called a `coarse interval', and the intervals
of~\cite{LaMe} `fine intervals'.)

\begin{lemma}   \label{lemma:tfae-cif}
The following conditions on a category $\A$ are equivalent:
\begin{enumerate}
\item   \label{part:tfae-cif-int}
for all $a, b \in \A$, the patch $\patch{\A}{a}{b}$ is a finite category

\item   \label{part:tfae-cif-ff}
$\A$ is finely finite and has finite homsets

\item   \label{part:tfae-cif-elem}
for all $a, b \in \A$, the set $\{ c \in \A \such \text{there exist
maps } a \to c \to b \}$ is finite, and $\A$ has finite homsets.
\done
\end{enumerate}
\end{lemma}

A category $\A$ is \demph{patch-finite} if it satisfies the equivalent
conditions of Lemma~\ref{lemma:tfae-cif}.  For example, a poset $A$ is
patch-finite if and only if it is locally finite
(Example~\ref{eg:poset-fine}).  

We have met three finiteness conditions: coarse, patch and fine.  They are
not \latin{ad hoc}.  To see how they arise systematically, recall from
Section~\ref{sec:fun} that the inclusions $\One \incl \Two \incl \Set$ give
rise to three monads $Q$ on $\Cat$, namely, $C$, $P$ and the identity.  To
make $\iaf{k}{(Q\A)}$ into an algebra, we require $Q\A$ to be finely
finite.  To furnish $\iaf{k}{(Q\A)}$ with a zeta function, we want to
transport the zeta function of $\iaf{k}{\A}$ along the unit map $\A \to Q\A$,
and for that we require $\A \to Q\A$ to have finite fibres.  So: to make
the basic definitions, we require $Q\A$ to be finely finite and $\A \to
Q\A$ to have finite fibres.

In the case $Q = \id$, this just says that $\A$ is finely finite.  In the case
$Q = P$, it says that $\A$ is patch-finite (by
Lemma~\ref{lemma:tfae-cif}(\ref{part:tfae-cif-elem})).  In the case $Q = C$,
it says that $\A$ is coarsely finite.  In fact, the three conditions are
successively stronger: 
\[
\text{coarsely finite }
\implies
\text{ patch-finite }
\implies
\text{ finely finite}.
\]

Let $\A$ be a patch-finite category.  Then the algebra $\iap{k}{\A} =
\iaf{k}{(P\A)}$ is defined and the induced map $\Sigma\cln \iaf{k}{\A} \to
\iap{k}{\A}$ is a homomorphism; the \demph{coarse zeta function} $\zeta_\A \in
\iap{k}{\A}$ is the image under $\Sigma$ of $\zeta_\A \in \iaf{k}{\A}$.
Explicitly, $\iap{k}{\A}$ is the submodule of $\iac{k}{\A}$ specified
in~\eqref{eq:iap-sub}, and the product on $\iap{k}{\A}$ is given by
\[
(\alpha\conv\beta)(a, b)
=
\sum_{c \in \patch{\A}{a}{b}} \alpha(a, c) \, \beta(c, b)
\]
($\alpha, \beta \in \iap{k}{\A}$, $a, b \in \A$).  (There is no product
defined on the larger module $\iac{k}{\A}$ unless $\A$ is finite.)  As before,
the zeta function is given explicitly by $\zeta_\A(a, b) = \card{(\A(a, b))}
\in k$.

For example, when $A$ is a locally finite poset, $\iap{k}{A}$ is the
classical incidence algebra.

\begin{defn}    \label{defn:patch-mi}
A patch-finite category $\A$ has \demph{coarse M\"obius inversion}
if $\zeta_\A \in \iap{k}{\A}$ is invertible.  In that case, its \demph{coarse
M\"obius function} is $\mu_\A = \zeta_\A^{-1} \in \iap{k}{\A}$.
\end{defn}

\latin{Prima facie}, we should have used different terminology: 
`patch M\"obius inversion/function'.  After all, when $\A$ is a finite
category, $\iap{k}{\A}$ is in general a \emph{proper} subalgebra of
$\iac{k}{\A}$, so one might think that it would be easier to invert $\zeta_\A$
in $\iac{k}{\A}$ than in $\iap{k}{\A}$.  It is a nontrivial fact that it makes
no difference (Corollary~\ref{cor:no-diff}).  Definition~\ref{defn:patch-mi}
is therefore consistent with the definitions for finite categories.

\begin{lemma}   \label{lemma:trans}
Let $\A$ be a finite category, $n \geq 0$, and $a_0, \ldots, a_n \in \A$.
Then
\[
\zeta(a_0, a_1) \cdots \zeta(a_{n-1}, a_n) \neq 0
\ \implies\ 
\zeta(a_0, a_n) \neq 0.
\]
\end{lemma}

\begin{proof}
Since $k$ has characteristic zero, an equivalent statement is that if the set
$\A(a_0, a_1) \times \cdots \times \A(a_{n-1}, a_n)$ is nonempty then so is
$\A(a_0, a_n)$.  But since $\A$ is a category, there is a map from the first
set to the second, and the result follows.  \done
\end{proof}

\begin{thm}     \label{thm:zm-van}
Let $\A$ be a finite category with coarse M\"obius inversion over $k$.
Let $a, b \in \A$.  Then $\zeta_\A(a, b) = 0 \implies \mu_\A(a, b) = 0$.
\end{thm}

\begin{proof}
In the terminology of Appendix~\ref{app:zeros}, Lemma~\ref{lemma:trans} states
that $\zeta_\A$ is transitive.  The result follows from
Theorem~\ref{thm:zeros} on inverse matrices.  
\done
\end{proof}

\begin{cor}     \label{cor:no-diff}
Let $\A$ be a finite category.  The coarse zeta function of $\A$ is
invertible in $\iap{k}{\A}$ if and only if it is invertible in $\iac{k}{\A}$.
\done
\end{cor}

Consider, for example, a finite poset $A$.  The algebra $\iaf{k}{A} =
\iap{k}{A}$ consists of the $k$-valued functions on pairs $(a, b) \in A
\times A$ such that $a \leq b$, whereas the algebra $\iac{k}{A}$ consists
of the $k$-valued functions on \emph{all} pairs $(a, b)$.  When $k$ is a
ring, the zeta function is always invertible in $\iaf{k}{A}$ (and therefore
in $\iac{k}{A}$), by the formula in Example~\ref{eg:poset-fine}.  But for
other rigs, it might not be invertible in $\iaf{k}{A}$, and
Corollary~\ref{cor:no-diff} then implies that it is not invertible in the
larger algebra $\iac{k}{A}$ either.  These and earlier remarks tell us, in
short, that the results on categorical M\"obius inversion presented here
give no more for posets than was already known to Rota \latin{et al}.

When a patch-finite category has coarse M\"obius inversion, its M\"obius
function is determined `locally', that is, patchwise:

\begin{propn}
Let $\A$ be a patch-finite category.  Then $\A$ has coarse M\"obius inversion
if and only if each patch $\patch{\A}{a}{b}$ does.  In that case, the coarse
M\"obius function of each patch $\patch{\A}{a}{b}$ is the restriction of that
of $\A$.
\end{propn}
This was stated without proof in the case of finite $\A$ as Corollary~4.3
of~\cite{ECC}. 

\begin{proof}
First suppose that $\A$ has coarse M\"obius inversion, with coarse M\"obius
function $\mu \in \iap{k}{\A}$.  Let $a, b \in \A$.  We have to prove that for
all $x, y \in \patch{\A}{a}{b}$,
\begin{equation}        \label{eq:mz-local}
\sum_{z \in \patch{\A}{a}{b}} \mu(x, z) \, \zeta(z, y) 
=
\delta(x, y),
\end{equation}
and similarly with $\mu$ and $\zeta$ interchanged.  By definition of $\mu$,
this equation holds when the sum is taken over all $z \in \patch{\A}{x}{y}$.
But $\patch{\A}{x}{y} \sub \patch{\A}{a}{b}$, and conversely if $z \in
\patch{\A}{a}{b}$ with $\mu(x, z) \zeta(z, y) \neq 0$ then $z \in
\patch{\A}{x}{y}$ (since $\mu \in \iap{k}{\A}$).  This
gives~\eqref{eq:mz-local}.

Conversely, suppose that for each $a, b \in \A$, the patch $\patch{\A}{a}{b}$
has coarse M\"obius inversion, with coarse M\"obius function $\mu_{a, b}$.
Define $\mu \in \iap{k}{\A}$ by
\[
\mu(a, b)
=
\begin{cases}
\mu_{a, b}(a, b)        &
\text{if } \A(a, b) \neq \emptyset   \\
0                       &
\text{otherwise.}
\end{cases}
\]
We prove that $\mu$ is the coarse M\"obius function of $\A$.  Indeed, let $a,
b \in \A$.  Then
\begin{equation}        \label{eq:mz-global}
\sum_{c \in \patch{\A}{a}{b}} \mu(a, c) \, \zeta(c, b)
=
\sum_{c \in \patch{\A}{a}{b}} \mu_{a, c}(a, c) \, \zeta(c, b).
\end{equation}
It is straightforward to show that $\patch{(\patch{\A}{a}{b})}{a}{c} =
\patch{\A}{a}{c}$ whenever $c \in \patch{\A}{a}{b}$, using composition.  So by
the first part of the proof (with $\patch{\A}{a}{b}$ playing the role of
$\A$), the coarse M\"obius function of $\patch{\A}{a}{c}$ is the restriction
of that of $\patch{\A}{a}{b}$.  The right-hand side of~\eqref{eq:mz-global} is
therefore unchanged if we replace $\mu_{a, c}(a, c)$ by $\mu_{a, b}(a, c)$,
and the result follows by definition of $\mu_{a, b}$.  \done
\end{proof}

\begin{examples}
\begin{enumerate}
\item Let $\Dinj$ be the category whose objects are the natural numbers and
whose maps $m \to n$ are the order-preserving injections $\{1, \ldots, m\} \to
\{1, \ldots, n\}$.  For $a, b \in \nat$, the patch $\patch{\Dinj}{a}{b}$ is the
full subcategory on $\{n \in \nat \such a \leq n \leq b\}$.  This is always
finite, so $\Dinj$ is patch-finite.  It has coarse M\"obius inversion:
$\zeta(m, n) = \binom{n}{m}$ and $\mu(m, n) = (-1)^{n - m}\binom{n}{m}$.
(Compare Example~1.2(c) of~\cite{ECC}.)

\item The same is true with surjections in place of injections; now $\zeta(m,
n) = \binom{m - 1}{n - 1}$ and $\mu(m, n) = (-1)^{m - n}\binom{m - 1}{n - 1}$.
\end{enumerate}
\end{examples}

We can now generalize Haigh's comparison theorem and Menni's
corollary:

\begin{thm}     \label{thm:comparison-pf}
Theorem~\ref{thm:comparison-fin} holds when $\A$ is merely patch-finite.
\done
\end{thm}

\begin{cor}     \label{cor:Menni-pf}
Corollary~\ref{cor:Menni} holds when $\A$ and $\A'$ are merely patch-finite.
\done
\end{cor}

Rota's seminal paper~\cite{Rota} on M\"obius inversion contained two `main
theorems'.  The first, Theorem~1, described the compatibility of M\"obius
functions across a Galois connection between posets.  It was generalized
in~\cite{ECC} to adjunctions between finite categories.  We now generalize it
further, to patch-finite categories.

\begin{propn}   \label{propn:Rota-main}
Let $\A$ and $\B$ be patch-finite categories with coarse M\"obius inversion.
Let $\oppair{\A}{\B}{F}{G}$ be functors with finite fibres, with $F$ left
adjoint to $G$.  Then for all $a \in \A$ and $b \in \B$,
\[
\sum_{a'\cln F(a') = b} \mu_\A(a, a')
=
\sum_{b'\cln G(b') = a} \mu_\B(b', b).
\]
\end{propn}

\begin{proof}
Exactly as for Proposition~4.4 of~\cite{ECC}.
\done
\end{proof}

\section{M\"obius inversion for enriched categories}
\label{sec:enr}

The theory of fine M\"obius inversion does not seem to generalize to enriched
categories in an obvious way, speaking as it does of \emph{individual}
morphisms.  Coarse M\"obius inversion, however, generalizes easily.  All one
needs is a notion of size for the objects of the enriching category.  In fact,
coarse M\"obius inversion for enriched categories has already been
used extensively in the case of metric spaces
(Example~\ref{egs:enr-mi}(\ref{eg:enr-mi-ms})).  

We confine ourselves to enriched categories with a finite number of objects,
although by imitating the previous section, the theory can also be set up for
infinitely many objects. 

Fix a monoidal category $\V = (\V, \otimes, I)$, a rig $k$, and a monoid
homomorphism
\[
\cde \cln (\obc{\V}/\!\iso, \otimes, I) \to (k, \cdot, 1)
\]
where the domain is the monoid of isomorphism classes of objects of $\V$.

Let $\A$ be a $\V$-category with finitely many objects.  The \demph{coarse
incidence algebra} $\iac{k}{\A}$ is defined exactly as in the non-enriched
case.  The \demph{coarse zeta function} $\zeta_\A \in \iac{k}{\A}$ is
given by 
\[
\zeta_\A(a, b)
=
\cd{\A(a, b)} \in k
\]
($a, b \in \A$).  If $\zeta_\A$ has an inverse in $\iac{k}{\A}$ then $\A$ has
\demph{coarse M\"obius inversion} over $k$, and its \demph{coarse M\"obius
function} is $\mu_\A = \zeta_\A^{-1} \in \iac{k}{\A}$.  

The assumption that $\cde$ is a monoid homomorphism was not needed in order to
make these definitions, but will be used in Proposition~\ref{propn:mu-mult}. 

\begin{examples}        \label{egs:enr-mi}
\begin{enumerate}
\item Taking $\V$ to be the category of finite sets, with $\cd{X} = n\cdot 1
\in k$ when $X$ is an $n$-element set, we recover the definitions for
non-enriched categories. 

\item Take $\V$ to be the category $\Two = (0 \to 1)$ with $\min$ as tensor
product.  Take $k = \integers$, and put $\cd{0} = 0$ and $\cd{1} = 1$.  Then a
$\V$-category is a poset, and every finite $\V$-category has coarse M\"obius
inversion (Example~\ref{eg:poset-coarse}).

\item   \label{eg:enr-mi-ms}
Let $\V$ be the poset $([0, \infty], \geq)$, with monoidal structure
given by addition.  As shown by Lawvere~\cite{LawvMSG}, a $\V$-category is a
generalized metric space.  Put $k = \reals$ and $\cd{x} = e^{-x}$ ($x \in [0,
\infty]$).  This gives a notion of M\"obius inversion for metric
spaces.  Most metric spaces have M\"obius inversion, in a sense made
precise by Proposition~2.2.6(i) of~\cite{MMS}.  For example, all finite
subspaces of Euclidean space do (Theorem~2.5.3 of~\cite{MMS}); more generally,
so do all finite subspaces of $L^p[0, 1]$ whenever $0 < p \leq 2$ (Theorem~3.6
of~\cite{MeckPDM}).  

The \demph{magnitude} of a metric space, defined for finite spaces as
$\sum_{a, b} \mu(a, b)$, is especially significant.  The definition extends 
to a large class of compact metric spaces~\cite{MMS, MeckPDM},
where its geometric meaning begins to emerge: to take the simplest example,
the magnitude of a straight line segment is one plus half its length.  Further
connections with geometric measure are established in \cite{MMS, AMSES,
MeckPDM, WillHCC, WillMSS}.

\item Let $\V$ be the category of finite-dimensional vector spaces with its
usual tensor product, let $k$ be any rig, and put $\cd{X} = (\dim X)\cdot 1
\in k$.  Then we obtain a notion of coarse M\"obius inversion for linear
categories.

\item Let $\V$ be the category of finite categories with Euler
characteristic~\cite{ECC}, made monoidal by cartesian product.  Let $k =
\rationals$, and put $\cd{\scat{X}} = \chi(\scat{X}) \in \rationals$.  (This
is a monoid homomorphism, by Proposition~2.6 of~\cite{ECC}.)  We obtain a
notion of coarse M\"obius inversion for (some) finite 2-categories.

\item Let $\V$ be the category $\FinSet^\nat$ of sequences of finite sets,
with 
$(X \otimes Y)_n = \sum_{p + q = n} X_p \times
Y_q$.  A $\V$-category is a category in which each map $f$ has a degree
$\deg(f) \in \nat$, such that for each $a$, $b$, there are only finitely many
maps $a \to b$ of each degree, and $\deg(g \of f) = \deg(f) + \deg(g)$.  Let
$k = \rationals\Lau{t}$, the ring of formal Laurent series over $\rationals$.
Put $\cd{X} = \sum_{n \in \nat} \card{X_n} \cdot t^n$.  We obtain a notion of
coarse M\"obius inversion for graded categories.

For example, let $G$ be a finite directed graph.  The free category $FG$ on
$G$ need not be finite, but is naturally $\V$-enriched: a map in $FG$ is a
path in $G$, with degree defined as length.  It has M\"obius inversion, as
follows.  Write $G_0$ and $G_1$ for the sets of vertices and edges of $G$,
and, for $a, b \in G$, write $G(a, b)$ for the set of edges from $a$ to
$b$.  Define $\zeta_G \in \iac{k}{\A}$ by $\zeta_G(a, b) = \card{G(a, b)}$.
Then $\zeta_{FG} = \sum_{n \in \nat} (\zeta_G \cdot t)^{\conv n}$ and
$\mu_{FG} = \delta - \zeta_G \cdot t$.  It follows that $\sum_{a, b}
\mu_{FG}(a, b) = \card{G_0} - \card{G_1} \cdot t$.  (For instance, if $G$
has just one vertex and $m$ edges then $FG$ is the free monoid on $m$
generators and $\sum_{a, b} \mu_{FG}(a, b) = 1 - mt$.)  When $t = 1$, this
is the Euler characteristic of $G$; compare Proposition~2.10 of~\cite{ECC}.

\end{enumerate}
\end{examples}

Coarse M\"obius inversion interacts well with tensor product of enriched
categories.  Assume now that $\V$ is symmetric, so that the
tensor product of $\V$-categories is defined.  The following result
generalizes Lemma~1.13(b) of~\cite{ECC}, and is proved using the
multiplicative property of $\cde$.

\begin{propn}   \label{propn:mu-mult}
Let $\A$ and $\B$ be $\V$-categories with finite object-sets.  If $\A$ and
$\B$ have coarse M\"obius inversion over $k$ then so does $\A \otimes \B$,
with
\[
\mu_{\A \otimes \B} ((a, b), (a', b'))
=
\mu_\A(a, a') \mu_\B(b, b')
\]
($a, a' \in \A$, $b, b' \in \B$).
\done
\end{propn}

There is a similar result on coproducts, generalizing Lemma~1.13(a)
of~\cite{ECC}.  (Compare also Proposition~1.4.4 of~\cite{MMS}.)  Our
generalization of Rota's main theorem (Proposition~\ref{propn:Rota-main}) also
extends easily to the enriched setting.

\section{Functoriality revisited}
\label{sec:fun-re}

The incidence algebra construction is functorial in two ways: covariant and
contravariant.  We have already used the covariant functoriality.  Here we
examine its contravariant counterpart.  We then show that the two types of
functoriality interact well enough that they can, in fact, be unified into a
single functor.  

A functor $F\cln \A \to \B$ has \demph{unique lifting of factorizations}, or
is \demph{ULF}, if whenever $f$ is a map in $\A$ and $Ff = g_2 \of g_1$ in
$\B$, there are unique maps $f_1, f_2$ in $\A$ such that $f_2 \of f_1 = f$,
$Ff_1 = g_1$ and $Ff_2 = g_2$:
\[
\begin{array}{c}
\xymatrix{
        &\cdot \ar@{.>}[dr]^{f_2}       &       \\
a \ar@{.>}[ur]^{f_1} \ar[rr]_f  &       &b
}
\end{array}
\qquad
\stackrel{F}{\longmapsto}
\qquad
\begin{array}{c}
\xymatrix{
                                &\cdot \ar[dr]^{g_2}   &       \\
Fa \ar[ur]^{g_1} \ar[rr]_{Ff}   &                       &Fb.
}
\end{array}
\]
This definition is implicit in Th\'eor\`eme~4.1 of~\cite{CLL}, and is made
explicit in Section~4 of~\cite{LawvTCS}.  Appendix~\ref{app:pbh}
places the ULF concept into an abstract context.

Let $F\cln \A \to \B$ be a functor between finely finite categories.  For each
rig $k$, there is an induced $k$-linear map
\[
F^*\cln \iaf{k}{\B} \to \iaf{k}{\A}
\]
defined by 
\[
(F^* \beta)(f) = \beta(Ff)
\]
where $\beta \in \iaf{k}{\B}$ and $f$ is a map in $\A$.  It is a fact that
$F$ is ULF if and only if $F^*$ is an algebra homomorphism for all rings
$k$: again, this is implicit in Th\'eor\`eme~4.1 of~\cite{CLL}, and it is made
explicit in Theorem~9.21 of~\cite{LaMe}.  Our Proposition~\ref{propn:fun} is
a covariant companion of this fact.

For example, whenever $X$ is an object of a category $\cat{C}$, the
forgetful functor $X/\cat{C} \to \cat{C}$ is ULF.  Lawvere (Section~4
of~\cite{LawvTCS}) and Lawvere and Menni (Example~9.22 of~\cite{LaMe})
point out the following.  When $\cat{C}$ is the additive monoid of natural
numbers, viewed as a one-object category $(\nat, +, 0)$, this is the
functor $(\nat, \leq) \to (\nat, +, 0)$ sending the map $m \to n$ in
$(\nat, \leq)$ to the map $n - m$ in $(\nat, +, 0)$, whenever $m \leq n$.
It induces an algebra homomorphism $\iaf{k}{(\nat, +, 0)} \to
\iaf{k}{(\nat, \leq)}$, thus relating the monoid M\"obius inversion of
Cartier and Foata~\cite{CaFo} to the poset M\"obius inversion of Rota et
al.

The class of ULF functors is closed under composition, so there is a category
$\Catulf$ of finely finite categories and ULF functors.  There is then a
functor $\Catulf^\op \to k\hyph\Alg$ defined by $\A \mapsto \iaf{k}{\A}$ and
$F \mapsto F^*$.

The covariant and contravariant constructions are linked by a result with a
strong formal resemblance to the Beck--Chevalley theorem.

\begin{thm}   \label{thm:BC}
Let
\[
\xymatrix{
\scat{D} \ar[r]^{F'} \ar[d]_{G'}        &
\scat{B} \ar[d]^G       \\
\scat{A} \ar[r]_F       &
\scat{C}        
}
\]
be a pullback square in $\Cat$.  Suppose that all four categories are finely
finite, $F$ is ULF, and $G$ is bijective on objects and has finite fibres.
Then $F'$ is ULF, $G'$ is bijective on objects and has finite fibres, and the
square
\[
\xymatrix{
\iaf{k}{\scat{D}} \ar[r]^{F'_!} &
\iaf{k}{\scat{B}}       \\
\iaf{k}{\scat{A}} \ar[r]_{F_!} \ar[u]^{G'^*}    &
\iaf{k}{\scat{C}} \ar[u]_{G^*}
}
\]
commutes for all rigs $k$.
\end{thm}

\begin{proof}
That $F'$ is ULF follows from the fact that the pullback of an ULF functor
along an arbitrary functor is again ULF, which can be checked directly and
also follows from Proposition~\ref{propn:cart-implies-stable}.  That $G'$ is
bijective on objects and has finite fibres is straightforward.  Now let $\alpha
\in \iaf{k}{\A}$ and $g \in \arr{\B}$.  We have
\[
(G^* F_! \alpha)(g)
=
(F_! \alpha)(Gg)
=
\sum_{f \in \arr{\A}\cln Ff = Gg} \alpha(f).
\]
On the other hand, 
\[
(F'_! G'^* \alpha)(g)
=
\sum_{h \in \arr{\scat{D}}\cln F'h = g} \alpha(G' h)
=
\sum_{f \in \arr{\A}\cln Ff = Gg} \alpha(f)
\]
since the square is a pullback.
\done
\end{proof}

We can now unify the two types of functoriality for incidence algebras.  

The bicategory of spans in $\Cat$~\cite{BenIB} has a sub-bicategory
$\Cat_!^*$, defined as follows.  The objects are the finely finite categories.
The 1-cells from $\A$ to $\B$ are the spans
\begin{equation}        \label{eq:span}
\xymatrix{
\A &&
\scat{C} \ar[ll]_F^{\text{ULF}} \ar[rr]^G_{\text{BO, FF}} &&
\B
}
\end{equation}
in which $F$ is ULF and $G$ is bijective on objects and has
finite fibres.  The 2-cells are the isomorphisms.  We may also view
$k\hyph\Alg$ as a bicategory, with only identity 2-cells.

\begin{cor}
There is a strict functor $\Cat_!^* \to k\hyph\Alg$ defined on objects by $k
\mapsto \iaf{k}{\A}$ and on 1-cells by sending~\eqref{eq:span} to the
composite homomorphism $k\A \toby{F^*} k\scat{C} \toby{G_!} k\B$.
\end{cor}

\begin{proof}
Theorem~\ref{thm:BC} implies that composition is preserved, and the rest is
trivial. 
\done
\end{proof}

A cruder version of the same result uses the \emph{category} $\Cat_!^*$ whose
maps are the \emph{isomorphism classes} of spans~\eqref{eq:span}.  We still
obtain a functor $\Cat_!^* \to k\hyph\Alg$.

\section{The M\"obius categories of Leroux}
\label{sec:McL}

In the work of Leroux et al.~\cite{CLL, Lero}, a central role is played by the
`M\"obius categories'.  (Beware that Haigh~\cite{Haig} uses the same term
differently.)  A category is \demph{M\"obius} if it is finely finite and
satisfies the equivalent conditions of the following theorem.

\begin{thm}[Content--Lemay--Leroux]       \label{thm:tfae-cll}
Let $\A$ be a finely finite category.  The following conditions on $\A$ are
equivalent:
\begin{enumerate}
\item   \label{part:tfae-cll-iii}
Every isomorphism or idempotent in $\A$ is an identity.

\item   \label{part:tfae-cll-fin}
Each map in $\A$ can be expressed as a composite of a finite sequence of
non-identity maps in only finitely many ways.

\item   \label{part:tfae-cll-mi}
For all rings $k$, an element $\alpha \in \iaf{k}{\A}$ is invertible if
and only if $\alpha(1_a) \in k$ is invertible for all $a \in \A$.
\end{enumerate}
\end{thm}

\begin{proof}
This is nearly Th\'eor\`eme~1.1 of~\cite{CLL}, except that where we have
condition~(\ref{part:tfae-cll-iii}), they have the conjunction of two
conditions: (a)~if $g \of f = 1_a$ in $\A$ then $g = f = 1_a$, and (b)~if $h$
is an endomorphism in $\A$ with $h^m = h^n$ for some natural numbers $m \neq
n$ then $h$ is an identity.

Certainly~(a) and~(b) together imply~(\ref{part:tfae-cll-iii}).  The converse
does not seem to have been stated completely explicitly before, although
essentially it goes back to~\cite{Lero} (and it is proved for finite
categories in Proposition~3.5 of~\cite{LaMe}).  Suppose
that~(\ref{part:tfae-cll-iii}) holds.  For~(a), if $g \of f = 1_a$ in $\A$
then $f\of g$ is idempotent, so $f \of g$ is an identity, so $f$ and $g$ are
isomorphisms and therefore identities.  For~(b), suppose that $h^n =
h^{n + k}$ for some $n, k \geq 1$; then $h^{nk}$ is idempotent, so $h^{nk} =
1$, which by~(a) implies that $h$ is an identity.  \done
\end{proof}

Being M\"obius is a strictly stronger condition than having fine M\"obius
inversion over all rings.  It is stronger by~(\ref{part:tfae-cll-mi}), and
\emph{strictly} stronger by the following example.

\begin{example}
Let $\A$ be the category freely generated by objects and maps
$\oppair{a}{b}{s}{i}$ subject to $s i = 1_b$.  It is easily shown that $\A$
has fine M\"obius inversion over all rings (with $\mu(1_a) = 1$, $\mu(1_b) =
2$, $\mu(s) = \mu(i) = -1$, and $\mu(i s) = 0$).  But $\A$ is not M\"obius,
since it contains the nontrivial idempotent $is$.
\end{example}

This example can be viewed as follows.  By
Theorem~\ref{thm:tfae-cll}(\ref{part:tfae-cll-fin}), every subcategory (full
or not) of a M\"obius category is M\"obius.  In particular, every subcategory
of a M\"obius category has fine M\"obius inversion over all rings.  However,
$\A$ contains the subcategory $\B$ consisting of the object $a$, the identity
$1_a$, and the idempotent $e = is \neq 1_a$, which does \emph{not} have fine
M\"obius inversion over all rings: the M\"obius function would have to
satisfy $2\mu_\B(e) = -1$.

So, having a subcategory without fine M\"obius inversion is an obstruction to
being M\"obius.  The main result of this section is that it is the \emph{only}
obstruction.  

\begin{thm}     \label{thm:mc-mi}
Let $\A$ be a finely finite category.  The following conditions on $\A$ are
equivalent: 
\begin{enumerate}
\item   \label{part:mc-mi-mc}
$\A$ is M\"obius
\item   \label{part:mc-mi-sub-all}
every subcategory of $\A$ has fine M\"obius inversion over every ring
\item   \label{part:mc-mi-sub-Z}
every subcategory of $\A$ has fine M\"obius inversion over $\integers$.
\end{enumerate}
\end{thm}

\begin{proof}
We have just seen that~(\ref{part:mc-mi-mc}) $\implies$
(\ref{part:mc-mi-sub-all}), and~(\ref{part:mc-mi-sub-all}) $\implies$
(\ref{part:mc-mi-sub-Z}) trivially.  Now suppose~(\ref{part:mc-mi-sub-Z}).  We
prove condition~(\ref{part:tfae-cll-iii}) of Theorem~\ref{thm:tfae-cll}.

Let $i\cln a \to b$ be an isomorphism in $\A$.  Since $\A$ is finely finite,
$1_a$ has only finitely many factorizations into two factors; write them as
$g_1 \of f_1, \ldots, g_n \of f_n$.  Then the distinct factorizations of $i$
are $(ig_1) \of f_1, \ldots, (ig_n) \of f_n$.  But $\A$ itself has fine
M\"obius inversion over $\integers$, and
\[
\sum_{r = 1}^n \mu_\A(f_r) = \delta(1_a),
\qquad
\sum_{r = 1}^n \mu_\A(f_r) = \delta(i),
\]
so $\delta(i) = \delta(1_a) = 1 \in \integers$.  Hence $i$ is an
identity.

Now let $e\cln a \to a$ be an idempotent in $\A$.  As above, the subcategory
consisting of the object $a$ and the maps $1_a$ and $e$ can only have fine
M\"obius inversion over $\integers$ if $e = 1_a$.  \done
\end{proof}

Further characterizations of M\"obius categories can be found in~\cite{CLL,
LaMe, Lero}.

\appendix

\section{Zeros of the M\"obius function}
\label{app:zeros}

To extend the definition of coarse M\"obius inversion to categories with
infinitely many objects, we made essential use of Theorem~\ref{thm:zm-van},
the proof of which depended in turn on a fact about matrices:
Theorem~\ref{thm:zeros} below.  Our task here is to prove this.  

The same result was proved in the case $k = \rationals$ as Theorem~4.1
of~\cite{ECC}.  For arbitrary rigs, the proof is complicated by the need to
avoid subtraction.

Fix a rig $k$.  Write the $(i, j)$-entry of a matrix $X$ as $X_{ij}$.

\begin{defn}
An $n \times n$ matrix $Z$ over $k$ is \demph{transitive} if for all $p \geq
0$ and $i_1, \ldots, i_p \in \{1, \ldots, n\}$,
\[
Z_{i_0 i_p} = 0
\ \implies\ 
Z_{i_0 i_1} Z_{i_1 i_2} \cdots Z_{i_{p-1} i_p} = 0.
\]
\end{defn}
The case $p = 0$ states that $Z_{ii} = 0 \implies 1 = 0$; that is, if
$k$ is nontrivial then $Z_{ii} \neq 0$.

For an $n \times n$ matrix $X$ over $k$, write
\[
\detp X 
= 
\sum_{\sigma \in A_n} 
\prod_{r = 1}^n X_{r, \sigma(r)},
\qquad
\detm X 
= 
\sum_{\sigma \in S_n \without A_n} 
\prod_{r = 1}^n X_{r, \sigma(r)}.
\]
Thus, $\det X = \detp X - \detm X$.  Let $\adjp X$ and $\adjm X$ be the $n
\times n$ matrices with entries
\[
\adjp_{ij} X
=
\sum_{\sigma \in A_n \cln \sigma(j) = i} 
\ 
\prod_{r \neq j} X_{r, \sigma(r)},
\qquad
\adjm_{ij} X
=
\sum_{\sigma \in S_n \without A_n \cln \sigma(j) = i} 
\ 
\prod_{r \neq j} X_{r, \sigma(r)}.
\]
Thus, $\adjp X - \adjm X$ is the adjugate (classical adjoint) $\adj
X$. 

\begin{lemma}   \label{lemma:van-ids}
The following identities hold, for $n \times n$ matrices $X$ and $Y$ over $k$.
\begin{enumerate}
\item   \label{part:van-ids-id}
$\detp I = 1$ and $\detm I = 0$.

\item   \label{part:van-ids-det}
$
(\detp X)(\detp Y) + (\detm X)(\detm Y) + \detm(XY) 
=
$\\
$
(\detp X)(\detm Y) + (\detm X)(\detp Y) + \detp(XY).
$

\item   \label{part:van-ids-adj}
$
X(\adjp X) + (\detm X) I 
=
X(\adjm X) + (\detp X) I.
$
\end{enumerate}
\end{lemma}

\begin{proof}
Part~(\ref{part:van-ids-id}) is immediate.  For~(\ref{part:van-ids-det}),
first note that the general identity $\det(XY) = (\det X)(\det Y)$ can be
regarded as an identity in the ring of polynomials over $\integers$ in $2n^2$
variables.  Substituting $\det = \detp - \detm$ gives the equation shown,
which is again an identity in this polynomial ring.  But all coefficients are
nonnegative, so it is also an identity in the rig of polynomials over $\nat$ in
$2n^2$ variables.  The result follows.  Part~(\ref{part:van-ids-adj}) is
proved similarly, using the identity $X(\adj X) = (\det X)I$ and the fact
that $\adj = \adjp - \adjm$.  \done
\end{proof}

\begin{lemma}   \label{lemma:van-minor}
Let $Z$ be an invertible, transitive $n \times n$ matrix over $k$.  Suppose
that $Z_{1n} = 0$.  Then:
\begin{enumerate}
\item   \label{part:van-minor-det}
Both $(\detp Z)(Z^{-1})_{1n}$ and $(\detm Z)(Z^{-1})_{1n}$ have additive
inverses in $k$.

\item   \label{part:van-minor-adj}
$\adjp_{1n} Z = \adjm_{1n} Z = 0$.
\end{enumerate}
\end{lemma}

\begin{proof}
First I claim that if $\sigma \in S_n$ with $\sigma(n) = 1$ then $\prod_{r =
1}^{n - 1} Z_{r, \sigma(r)} = 0$.  To prove this, choose the least $p \geq 1$
such that $\sigma^p(1) = 1$.  We have $\sigma^{p - 1}(1) = n$, and the numbers
$1, \sigma(1), \ldots, \sigma^{p - 2}(1)$ are all distinct and less than $n$,
so
\[
Z_{1, \sigma(1)} Z_{\sigma(1), \sigma^2(1)} \cdots 
Z_{\sigma^{p-2}(1), n}
\divides
Z_{1, \sigma(1)} Z_{2, \sigma(2)} \cdots Z_{n-1, \sigma(n-1)}.
\]
But by transitivity and the hypothesis $Z_{1n} = 0$, the left-hand side is
$0$, so the right-hand side is also $0$, as claimed.

For~(\ref{part:van-minor-det}), it is enough to prove that $(\prod_{r = 1}^n
Z_{r, \sigma(r)})(Z^{-1})_{1n}$ has an additive inverse for each $\sigma \in
S_n$.  When $\sigma(n) = 1$, this follows from the claim.  Suppose, then, that
$\sigma(n) \neq 1$.  We have
\[
\sum_{i = 1}^n Z_{\sigma^{-1}(1), i} (Z^{-1})_{in}
=
I_{\sigma^{-1}(1), n}
=
0,
\]
so $Z_{\sigma^{-1}(1), 1} (Z^{-1})_{1n}$ has an additive inverse, and
the result follows.

Part~(\ref{part:van-minor-adj}) follows immediately from the claim.
\done
\end{proof}

\begin{thm}     \label{thm:zeros}
Let $Z$ be an invertible, transitive, $n \times n$ matrix over $k$.  Let $i, j
\in \{1, \ldots, n\}$.  Then
\[
Z_{ij} = 0
\implies
(Z^{-1})_{ij} = 0.
\]
\end{thm}

\begin{proof}
If $i = j$ then $Z_{ii} = 0$, so by transitivity, $k$ is trivial and the
result holds.  So we may suppose without loss of generality that $i = 1$ and
$j = n$.

Applying Lemma~\ref{lemma:van-ids}(\ref{part:van-ids-adj}) to $Z$, then
premultiplying by $Z^{-1}$, we have
\[
\adjp Z + (\detm Z) Z^{-1} 
=
\adjm Z + (\detp Z) Z^{-1}.
\]
Now taking the $(1, n)$ entries on each side and using
Lemma~\ref{lemma:van-minor}(\ref{part:van-minor-adj}), we have
\begin{equation}        \label{eq:van-detpm}
(\detm Z)(Z^{-1})_{1n} 
=
(\detp Z)(Z^{-1})_{1n}.
\end{equation}
On the other hand, we may take $X = Z$ and $Y = Z^{-1}$ in
Lemma~\ref{lemma:van-ids}(\ref{part:van-ids-det}), which, with the aid of
Lemma~\ref{lemma:van-ids}(\ref{part:van-ids-id}), gives
\begin{equation}        \label{eq:van-big}
(\detp Z)(\detp Z^{-1}) + (\detm Z)(\detm Z^{-1}) 
=
(\detp Z)(\detm Z^{-1}) + (\detm Z)(\detp Z^{-1}) + 1.
\end{equation}
Multiply~\eqref{eq:van-big} by $(Z^{-1})_{1n}$ on each side.
By~\eqref{eq:van-detpm}, the result is an equation of the form $\lambda =
\lambda + (Z^{-1})_{1n}$, where, by
Lemma~\ref{lemma:van-minor}(\ref{part:van-minor-det}), $\lambda \in k$ has an
additive inverse.  Hence $(Z^{-1})_{1n} = 0$.  \done
\end{proof}

\section{Pullback-homomorphisms}
\label{app:pbh}

Here we place the notion of ULF functor into an abstract context.  In doing
so, we discover a new analogy between ULF functors and local homeomorphisms.

\begin{defn}
Let $\T = (T, \eta, \mu)$ be a monad on a category $\E$.  A homomorphism 
\begin{equation}
\label{eq:alg-homm}
\begin{array}{c}
\xymatrix{
TA \ar[r]^{Tf} \ar[d]   &TB \ar[d]     \\
A \ar[r]^f              &B
}
\end{array}
\end{equation}
of $\T$-algebras is a \demph{pullback-homomorphism} if the
square~\eqref{eq:alg-homm} is a pullback.
\end{defn}

\begin{propn}
Let $\T$ be the free category monad on the category of directed graphs.  Then
the pullback-homomorphisms of $\T$-algebras are precisely the ULF functors.
\end{propn}

\begin{proof}
Let $F\cln \A \to \B$ be a functor between small categories, regarded as a
homomorphism of $\T$-algebras.  Write $\A_n$ for the set of paths $a_0
\toby{f_1} \cdots \toby{f_n} a_n$ in $\A$, and similarly $\B_n$.  Since limits
in a presheaf category are computed pointwise, $F$ is a pullback-homomorphism
if and only if the squares
\[
\xymatrix{
\A_0 \ar[r] \ar[d]_1      &\B_0 \ar[d]^1    \\
\A_0 \ar[r]               &\B_0
}
\qquad\qquad
\xymatrix{
\sum_{n \in \nat} \A_n \ar[r] \ar[d]_{\circ}    &
\sum_{n \in \nat} \B_n \ar[d]^{\circ}   \\
\A_1 \ar[r]     &\B_1
}
\]
are pullbacks in $\Set$.  (Here $\sum$ denotes coproduct.)  The left-hand
square certainly is, and the right-hand square is a pullback if and only if
\[
\xymatrix{
\A_n \ar[r] \ar[d]_\circ    &\B_n \ar[d]^\circ  \\
\A_1 \ar[r]     &\B_1
}
\]
is a pullback for each $n \in \nat$.  This reduces by induction to the cases
$n = 0$ and $n = 2$.  For the $n = 2$ square to be a pullback is precisely the
ULF property.  The $n = 0$ square is a pullback if and only if $F$ reflects
identities; but this is always true if $F$ is ULF.
\done
\end{proof}

Pullback-homomorphisms have a three-for-two property: given homomorphisms
$\obdot \toby{f} \obdot \toby{g} \obdot$ with $g$ a pullback-homomorphism, $g
\of f$ is a pullback-homomorphism if and only if $f$ is.  This follows from
the elementary properties of pullbacks, and applies in particular to ULF
functors. 

Here are the pullback-homomorphisms for some other monads.  Proofs are
omitted. 

\begin{examples}        \label{egs:pbh}
\begin{enumerate}
\item   \label{eg:pbh-gset}
Fix a group $G$ and consider the monad $G \times \dashbk$ on
$\Set$, whose algebras are $G$-sets.  Then every map of $G$-sets is a
pullback-homomorphism.  

\item \label{eg:pbh-gp} At the other extreme, when $\T$ is the free group
monad on $\Set$, the only pullback-homomorphisms are the isomorphisms.

\item   \label{eg:pbh-ptd}
Take the monad $1 + \dashbk$ on $\Set$, adjoining to each set a new element.
Its category of algebras is equivalent to the category of sets and partial
functions.  The pullback-homomorphisms are the \emph{total}
functions.  

\item \label{eg:pbh-lat} Let $\pset$ be the powerset monad on $\Set$.  Its
algebras are the complete lattices; the homomorphisms are the maps preserving
joins.  Among them, the pullback-homomorphisms are the injections whose images
are downwards closed.

\item   \label{eg:pbh-cart}
Let $\A$ be a small category.  The forgetful functor $\Set^\A \to
\Set^{\obc{\A}}$ is monadic.  So, writing $\T$ for the induced monad, the
homomorphisms of $\T$-algebras are the natural transformations between
functors $\A \to \Set$.  The pullback-homomorphisms are the cartesian
natural transformations: those for which every naturality square is a
pullback.  
\end{enumerate}
\end{examples}

The unwirable maps of Bowler~\cite{Bowl} provide further examples.

We have observed that the class of pullback-homomorphisms is closed under
composition.  For a general monad $\T$, it is not stable under pullback
(Example~\ref{eg:cpt-Hff-stable}); that is, the pullback of a
pullback-homomorphism along an arbitrary homomorphism need not be a
pullback-homomorphism.  However:

\begin{propn}   \label{propn:cart-implies-stable}
Let $\E$ be a category with pullbacks and $\T$ a monad on $\E$ whose functor
part preserves pullbacks.  Then the class of pullback-homomorphisms of
$\T$-algebras is stable under pullback along arbitrary homomorphisms.
\end{propn}

\begin{proof}
Elementary manipulation of pullbacks.
\done
\end{proof}

Since the free category monad on directed graphs preserves pullbacks, the
class of ULF functors is stable under pullback.
Proposition~\ref{propn:cart-implies-stable} also implies that the class of
pullback-homomorphisms is stable under pullback in
Examples~\ref{egs:pbh}(\ref{eg:pbh-gset}), (\ref{eg:pbh-ptd}),
(\ref{eg:pbh-cart}).  Furthermore, the same is true in
Examples~\ref{egs:pbh}(\ref{eg:pbh-gp}) and (\ref{eg:pbh-lat}), not by the
proposition but by the explicit description of pullback-homomorphisms given
there.  This covers all of our examples so far.

It is now useful to extend the terminology.

\begin{defn}
Let $\E$ be a category with pullbacks and $\T$ a monad on $\E$.  A
homomorphism $f\cln (A, \alpha) \to (B, \beta)$ of $\T$-algebras is a
\demph{stable pullback-homomorphism} if for every homomorphism $g\cln (C,
\gamma) \to (B, \beta)$ of $\T$-algebras, the pullback of $f$ along $g$ is a
pullback-homomorphism.
\end{defn}

Thus, the class $\cat{S}$ of stable pullback-homomorphisms is the largest
subclass of the pullback-homomorphisms that is stable under pullback along
arbitrary homomorphisms.  In all of our examples so far, every
pullback-homomorphism is stably so.

We finish with a suggestive example in which pullback-homomorphisms are
not stable under pullback.  I thank Mike Shulman for pointing it out. 

\begin{example} \label{eg:cpt-Hff-stable}
Let $\T$ be the ultrafilter monad on $\Set$, whose algebras are the compact
Hausdorff spaces.  It is shown in~\cite{CHJ} that not every
pullback-homomorphism of $\T$-algebras is stably so.  It is also shown that
the stable pullback-homomorphisms are precisely the local homeomorphisms.
\end{example}

According to Lawvere and Menni, `The definition of ULF-functor should be
compared with that of local homeomorphism' (\cite{LaMe}, p.230).  We now
have a general concept, stable pullback-homomorphism, of which both ULF
functors and local homeomorphisms (between compact Hausdorff spaces) are
special cases.  A further possibility, suggested by Joachim Kock, is that
there might also be a connection via the axiomatic notion of \'etale
map~\cite{JoMe}.

\ucontents{section}{References}

\end{document}

%% file: ab.tex
\usepackage{latexsym}
\usepackage{amssymb,amsmath}
\usepackage{mathrsfs}

\newcommand{\dashbk}{-}
\newcommand{\hyph}{\mbox{-}}
\newcommand{\ucontents}[2]{\addcontentsline{toc}{#1}{\numberline{}{#2}}}

\newcommand{\mb}[1]{\mathbf{#1}}
\newcommand{\mr}[1]{\mathrm{#1}}
\newcommand{\cat}[1]{\mathscr{#1}}
\newcommand{\fcat}[1]{\mb{#1}}
\newcommand{\such}{\:|\:}
\newcommand{\without}{\setminus}

\newcommand{\op}{\mr{op}}
\newcommand{\id}{\mr{id}}
\newcommand{\Alg}{\fcat{Alg}}
\newcommand{\Cat}{\fcat{Cat}}
\newcommand{\Set}{\fcat{Set}}
\newcommand{\integers}{\mathbb{Z}}

\newcommand{\oppairu}{\rightleftarrows}

\newcommand{\reals}{\mathbb{R}}
\newcommand{\rationals}{\mathbb{Q}}
\newcommand{\latin}[1]{#1}	
\newcommand{\demph}[1]{\textbf{\textup{#1}}}
\newcommand{\done}{\hfill\ensuremath{\Box}}
\newenvironment{prooflike}[1]{\begin{trivlist}\item\textbf{#1}\ }
{\end{trivlist}}
\newenvironment{proof}{\begin{prooflike}{Proof}}{\end{prooflike}}
\newcommand{\scat}[1]{\mathbf{#1}}
\newcommand{\iso}{\cong}
\newcommand{\nat}{\mathbb{N}}	
\newcommand{\of}{\,\raisebox{0.08ex}{\ensuremath{\scriptstyle\circ}}\,}
\newcommand{\sub}{\subseteq}
\newcommand{\cell}[4]{\put(#1,#2){\makebox(0,0)[#3]{#4}}}
\newcommand{\One}{\mathbf{1}}
\newcommand{\Two}{\mathbf{2}}
\newcommand{\toby}[1]{\stackrel{#1}{\to}}
\newcommand{\FinSet}{\fcat{FinSet}}
\DeclareMathOperator{\adj}{adj}
\newcommand{\incl}{\hookrightarrow}
\newcommand{\cln}{\colon}
\newcommand{\card}[1]{\# #1}
\newcommand{\Poset}{\fcat{Poset}}
\renewcommand{\implies}{\Rightarrow}
\newcommand{\iaf}[2]{#1 #2}
\newcommand{\iac}[2]{#1_\mathrm{c} #2}
\newcommand{\iap}[2]{#1_\mathrm{p} #2}
\newcommand{\codisc}[1]{C #1}

\newcommand{\poset}[1]{P #1}

\newcommand{\A}{\scat{A}}
\newcommand{\B}{\scat{B}}
\newcommand{\V}{\scat{V}}
\newcommand{\arr}[1]{#1_1}
\newcommand{\obc}[1]{#1_0}
\newcommand{\Catboff}{{\fcat{Cat}_!}}
\newcommand{\Catulf}{{\fcat{Cat}^*}}
\newcommand{\T}{\mathbf{T}}
\newcommand{\E}{\cat{E}}

\newcommand{\obdot}{\cdot}

\newcommand{\pset}{\mathscr{P}}
\newcommand{\detp}{{\textstyle\det^+}}
\newcommand{\detm}{{\textstyle\det^-}}
\newcommand{\adjp}{\adj^+}
\newcommand{\adjm}{\adj^-}
\newcommand{\divides}{\mid}
\newcommand{\Dinj}{\scat{D}_{\text{inj}}}

\newcommand{\cd}[1]{|#1|}
\newcommand{\cde}{\cd{\cdot}}
\newcommand{\conv}{*}
\newcommand{\gr}[1]{|#1|}
\newcommand{\Lau}[1]{(\!(#1)\!)}

\newcommand{\oppair}[4]{%
\xymatrix@1{%
#1 \ar@<.5ex>[r]^{#3} &{#2}\ar@<.5ex>[l]^{#4}%
}}
\newcommand{\patch}[3]{{[#2, #3]_{#1}}}

%% file: th.tex
\newtheorem{thm}{Theorem}[section]
\newtheorem{propn}[thm]{Proposition}
\newtheorem{lemma}[thm]{Lemma}
\newtheorem{cor}[thm]{Corollary}
\newtheorem{predefn}[thm]{Definition}
\newenvironment{defn}{\begin{predefn}\upshape}{\end{predefn}}
\newtheorem{preexample}[thm]{Example}
\newenvironment{example}{\begin{preexample}\upshape}{\end{preexample}}
\newtheorem{preexamples}[thm]{Examples}
\newenvironment{examples}{\begin{preexamples}\upshape}{\end{preexamples}}